\documentclass[12pt,reqno]{amsart}
\usepackage{amsmath, amsthm, amssymb, stmaryrd,mathrsfs}

\advance \topmargin by -\headheight
\advance \topmargin by -\headsep
     
\evensidemargin \oddsidemargin
\newtheorem{theorem}{Theorem}[section]
\newtheorem{lemma}[theorem]{Lemma}
\newtheorem{corollary}[theorem]{Corollary}

\theoremstyle{definition}

\theoremstyle{remark}

\numberwithin{equation}{section}

\newcommand{\mmod}[1]{\,\,(\text{\rm mod}\,\,#1)}

\def\bfa{{\mathbf a}}

\def\bfh{{\mathbf h}}

\def\bfn{{\mathbf n}}

\def\bfx{{\mathbf x}}
\def\bfy{{\mathbf y}}

\def\calI{{\mathcal I}}

 \def\Ktil{{\widetilde K}}

\def\Ktil{\widetilde K}

\def\dbN{{\mathbb N}}  
\def\dbR{{\mathbb R}}
\def\dbZ{{\mathbb Z}}\def\dbQ{{\mathbb Q}}

\def\grA{{\mathfrak A}}
\def\grB{{\mathfrak B}}

\def\grf{{\mathfrak f}}\def\grF{{\mathfrak F}}
\def\grg{{\mathfrak g}}

\def\grJ{{\mathfrak J}}
\def\grk{{\mathfrak k}} \def\grK{{\mathfrak K}}

\def\grm{{\mathfrak m}}\def\grM{{\mathfrak M}}
\def\grN{{\mathfrak N}}\def\grn{{\mathfrak n}}

\def\grS{{\mathfrak S}}\def\grP{{\mathfrak P}}
\def\grW{{\mathfrak W}}\def\grB{{\mathfrak B}}
\def\grK{{\mathfrak K}}\def\grp{{\mathfrak p}}

\def\grW{{\mathfrak W}}

\def\alp{{\alpha}} \def\bfalp{{\boldsymbol \alpha}} 
\def\bet{{\beta}}  \def\bfbet{{\boldsymbol \beta}}
\def\gam{{\gamma}} 
\def\bfgam{{\boldsymbol \gam}} 
\def\del{{\delta}} \def\Del{{\Delta}}

\def\tet{{\theta}} \def\bftet{{\boldsymbol \theta}}

\def\sig{{\sigma}}

 \def\Ome{{\Omega}}

\def\eps{\varepsilon}

\def\le{\leqslant} \def\ge{\geqslant}

\def\d{{\,{\rm d}}}

\begin{document}
\title[Hilbert-Kamke problem]{Subconvexity and the Hilbert-Kamke problem}
\author[Trevor D. Wooley]{Trevor D. Wooley}
\address{Department of Mathematics, Purdue University, 150 N. University Street, West 
Lafayette, IN 47907-2067, USA}
\email{twooley@purdue.edu}
\thanks{The author's work is supported by NSF grants DMS-2001549 and DMS-1854398.}
\date{}
\dedicatory{}
\begin{abstract}When $s\ge k\ge 3$ and $n_1,\ldots ,n_k$ are large natural numbers, 
denote by $A_{s,k}(\bfn)$ the number of solutions in non-negative integers $\bfx$ to the 
system
\[
x_1^j+\ldots +x_s^j=n_j\quad (1\le j\le k).
\]
Under appropriate local solubility conditions on $\bfn$, we obtain an asymptotic formula 
for $A_{s,k}(\bfn)$ when $s\ge k(k+1)$. This establishes a local-global principle in the 
Hilbert-Kamke problem at the convexity barrier. Our arguments involve minor arc estimates 
going beyond square-root cancellation.
\end{abstract}
\maketitle

\section{Introduction} In this memoir we consider the asymptotic formula for the number 
of representations in the Hilbert-Kamke problem, our goal being to derive conclusions at or 
below the classical convexity barrier. When $n_1,\ldots ,n_k$ are positive integers, we 
denote by $A_{s,k}(\bfn)$ the number of solutions in non-negative integers $\bfx$ to the 
system of Diophantine equations
\begin{equation}\label{1.1}
x_1^j+\ldots +x_s^j=n_j\quad (1\le j\le k).
\end{equation}
Motivated by his recent work on Waring's problem (see \cite{Hil1909}), Hilbert posed the 
problem of determining suitable conditions on $\bfn$ that would guarantee, for an 
appropriate function $H(k)$, the non-vanishing of $A_{s,k}(\bfn)$ for $s\ge H(k)$. This 
problem was taken up, first by Kamke \cite{Kam1921}, and subsequently by Mardzhanishvili 
\cite{Mar1937} using Vinogradov's methods. The precise nature of the local conditions on 
$\bfn$ that must be imposed are quite complicated to describe, and we defer further 
discussion on this issue until later in this section.\par

In order to outline the current state of play, we recall the mean value
\begin{equation}\label{1.2}
J_{t,k}(X)=\int_{[0,1)^k}\biggl| \sum_{1\le x\le X}e(\alp_1x+\ldots +\alp_kx^k)
\biggr|^{2t}\d\bfalp ,
\end{equation}
in which $e(z)$ denotes $e^{2\pi iz}$. A consequence of the main conjecture in 
{\it Vinogradov's mean value theorem} asserts that when $k\in \dbN$ and $t\ge k(k+1)/2$, 
then for each $\eps>0$ one has
\begin{equation}\label{1.3}
J_{t,k}(X)\ll X^{2t-\frac{1}{2}k(k+1)+\eps},
\end{equation}
with the implicit constant depending at most on $k$, $t$ and $\eps$. Experts in the 
Hardy-Littlewood method will perceive that, provided the upper bound (\ref{1.3}) is 
known to hold for $t\ge t_0(k)$, then one is able to derive an asymptotic formula 
for $A_{s,k}(\bfn)$ whenever $s>2t_0(k)$. This formula takes the shape
\begin{equation}\label{1.4}
A_{s,k}(\bfn)=\grJ_{s,k}(\bfn)\grS_{s,k}(\bfn)X^{s-k(k+1)/2}+o(X^{s-k(k+1)/2}),
\end{equation}
where $X=\max \{n_1,n_2^{1/2},\ldots ,n_k^{1/k}\}$ and $\grJ_{s,k}(\bfn)$ and 
$\grS_{s,k}(\bfn)$ are respectively the {\it singular integral} and {\it singular series} 
associated with this problem. For now we defer explicit definition of these quantities, as well 
as discussion of conditions ensuring that $1\ll \grJ_{s,k}(\bfn)\ll 1$ and 
$1\ll \grS_{s,k}(\bfn)\ll 1$.\par

Prior to 2010, in the classical version of the subject resolved by Arkhipov \cite{Ark1984}, 
the upper bound (\ref{1.3}) was known to hold for $t\ge (2+o(1))k^2\log k$. This delivers 
(\ref{1.4}) for $s\ge (4+o(1))k^2\log k$. The trivial lower bound $J_{t,k}(X)\gg X^t$ shows 
that (\ref{1.3}) cannot hold when $t<k(k+1)/2$, and hence $t_0(k)\ge k(k+1)/2$. 
Consequently, the application of conventional methods can at best establish the asymptotic 
formula (\ref{1.4}) when $s\ge k^2+k+1$. Decisive progress toward this convexity-limited 
bound was made by the author \cite[Theorem 9.2]{Woo2012a} via the efficient congruencing 
method, establishing (\ref{1.4}) for $s\ge 2k^2+2k+1$. Finally, advances in the theory of 
Vinogradov's mean value theorem (see \cite{BDG2016, Woo2016, Woo2019}) show that the 
conjectured estimate (\ref{1.3}) holds for $t\ge k(k+1)/2$, and hence the asymptotic 
formula (\ref{1.4}) holds for $s\ge k^2+k+1$.\par

Our goal in the present paper is to go beyond this convexity-limited conclusion by confirming 
(\ref{1.4}) for $s\ge k^2+k$. In preparation for the statement of our new conclusion, we 
first introduce the generating functions
\begin{equation}\label{1.5}
I(\bfbet;X)=\int_0^Xe(\bet_1\gam+\ldots +\bet_k\gam^k)\d\gam
\end{equation}
and
\begin{equation}\label{1.6}
S(q,\bfa)=\sum_{r=1}^q e_q(a_1r+\ldots +a_kr^k) ,
\end{equation}
in which $e_q(u)$ denotes $e^{2\pi i u/q}$. Then, with 
$X=\underset{1\le j\le k}{\max}n_j^{1/j}$, we define
\begin{equation}\label{1.7}
\grJ_{s,k}(\bfn)=\int_{\dbR^k}I(\bfbet ;1)^se\left(-\bet_1\frac{n_1}{X}-\ldots
 -\bet_k\frac{n_k}{X^k}\right)\d\bfbet 
\end{equation}
and
\begin{equation}\label{1.8}
\grS_{s,k}(\bfn)=\sum_{q=1}^\infty \sum_{\substack{1\le \bfa\le q\\ (q,a_1,\ldots ,a_k)=1
}}q^{-s}S(q,\bfa)^se_q(-a_1n_1-\ldots -a_kn_k).
\end{equation}

\begin{theorem}\label{theorem1.1} Suppose that $k\ge 3$ and $s\ge k^2+k$. Then, 
whenever $n_1,\ldots ,n_k$ are natural numbers sufficiently large in terms of $s$ and $k$, 
one has
\[
A_{s,k}(\bfn)=\grJ_{s,k}(\bfn)\grS_{s,k}(\bfn)X^{s-k(k+1)/2}+o(X^{s-k(k+1)/2}),
\]
in which $0\le \grJ_{s,k}(\bfn)\ll 1$ and $0\le \grS_{s,k}(\bfn)\ll 1$. If, moreover, the 
system (\ref{1.1}) possesses a non-singular real solution with positive coordinates, then 
$\grJ_{s,k}(\bfn)\gg 1$. Likewise, if the system (\ref{1.1}) possesses primitive non-singular 
$p$-adic solutions for each prime number $p$, then $\grS_{s,k}(\bfn)\gg 1$.
\end{theorem}

We remark that the singular integral $\grJ_{s,k}(\bfn)$ is known to converge absolutely 
for $s>\tfrac{1}{2}k(k+1)+1$, and that the singular series $\grS_{s,k}(\bfn)$ is known 
to converge absolutely for $s>\tfrac{1}{2}k(k+1)+2$ (see \cite[Theorem 1]{Ark1984} or 
\cite[Theorem 3.7]{AKC2004}).\par

The extensive theory of quadratic forms ensures that the situation with $k=2$ is simple to 
handle for $s\ge 5$, and indeed much can be said even when $s=4$. We remark also 
that two obvious local conditions are in play if one is to have solutions to the system 
(\ref{1.1}). First, by applying H\"older's inequality on the left hand side of (\ref{1.1}), one 
sees that for solutions to exist one must have
\[
n_l^{j/l}\le n_j\le s^{1-j/l}n_l^{j/l}\quad (1\le j\le l\le k).
\]
Second, from Fermat's theorem, for each prime $p$ one must have 
$n_l\equiv n_j\mmod{p}$ whenever $l\equiv j\mmod{p-1}$ and $1\le j\le l\le k$. The latter 
observation plainly impacts the $p$-adic solubility conditions associated with the system 
(\ref{1.1}), a matter rather complicated to analyse in full. Indeed $p$-adic solubility is not 
assured in general without at least $2^k-1$ variables being available. See the excellent 
accounts of Arkhipov \cite{Ark1984, Ark2006} and \cite[Chapter 8]{AKC2004} for a 
comprehensive account of such issues. However, if one is permitted to assume the existence 
of non-singular primitive solutions at each local completion of $\dbQ$, then one obtains an 
immediate corollary to Theorem \ref{theorem1.1} via the methods of \cite{Ark1984}.

\begin{corollary}\label{corollary1.2} Suppose the system (\ref{1.1}) has non-singular 
primitive solutions at each local completion of $\dbQ$, and $s\ge k(k+1)$. Then 
$A_{s,k}(\bfn)\gg n_1^{s-k(k+1)/2}$.
\end{corollary}

We establish Theorem \ref{theorem1.1} by applying the Hardy-Littlewood method, utilising 
an estimate for the contribution of the minor arcs going beyond square-root cancellation. 
This estimate is derived in \S2 by adapting the author's work on the asymptotic formula in 
Waring's problem \cite{Woo2012b}, the failure of translation-dilation invariance in the system 
(\ref{1.1}) permitting an additional variable to be extracted and then utilised. Familiar in the 
context of equations, this device is more challenging in the absence of an immediate 
Diophantine interpretation, as when deriving estimates restricted to the minor arcs in a 
Hardy-Littlewood dissection. We launch our application of the circle method in earnest in \S3 
with the discussion of a Hardy-Littlewood dissection. Then, in \S4, we convert the raw 
subconvex minor arc estimate extracted in \S3 from the work of \S2 into one more directly 
applicable to the proof of Theorem \ref{theorem1.1}. Following some pruning in \S5, the 
proof of Theorem \ref{theorem1.1} is concluded in \S6 with a brief analysis of the major arc 
contribution. Finally, in \S7, we make some remarks concerning other cognate problems to 
which our methods are applicable.\par   

Our basic parameter is $X$, a sufficiently large positive number. Whenever $\eps$ appears 
in a statement, either implicitly or explicitly, we assert that the statement holds for each 
$\eps>0$. In this paper, implicit constants in Vinogradov's notation $\ll$ and $\gg$ may 
depend on $\eps$, $k$ and $s$. We make use of vector notation in the form 
$\bfx=(x_1,\ldots,x_r)$, the dimension $r$ depending on the course of the argument. We 
also write $(a_1,\ldots ,a_s)$ for the greatest common divisor of the integers 
$a_1,\ldots ,a_s$, any ambiguity between ordered $s$-tuples and corresponding greatest 
common divisors being easily resolved by context. Finally, we write $\|\tet\|$ for 
$\min\{|\tet-m|:m\in \dbZ\}$.
 
\section{Mean value estimates via shifts} We begin by preparing the infrastructure required 
to describe our novel mean value estimate. When $k\ge 2$, define $f(\bfalp)=f_k(\bfalp;X)$ 
by putting
\begin{equation}\label{2.1}
f_k(\bfalp;X)=\sum_{0\le x\le X}e(\alp_1x+\alp_2x^2+\ldots +\alp_k x^k).
\end{equation}
Then, when $\bfh\in \dbZ^k$ and $\grB\subseteq \dbR$ is measurable, we introduce the 
mean value
\begin{equation}\label{2.2}
I_s(\grB;X;\bfh)=\int_\grB \int_{[0,1)^{k-1}}f_k(\bfalp;X)^se(-\bfalp \cdot \bfh)\d\bfalp ,
\end{equation}
in which $\bfalp \cdot \bfh =\alp_1h_1+\ldots +\alp_kh_k$ and ${\rm d}\bfalp$ denotes 
${\rm d}\alp_1\d\alp_2\cdots \d\alp_k$. Provided that we take 
$X\ge \max\{n_1,n_2^{1/2},\ldots ,n_k^{1/k}\}$, it follows by orthogonality that 
$A_{s,k}(\bfn)=I_s([0,1);X;\bfn)$. We make use of technology associated with Vinogradov's 
mean value theorem. With this in mind, when $t,k\in \dbN$, the parameter $X$ is a positive 
real number, and $\grB\subseteq \dbR$ is measurable, we define
\begin{equation}\label{2.3}
J_{t,k}^*(\grB;X)=\int_\grB \int_{[0,1)^{k-1}}|f_k(\bfalp;X)|^t\d\bfalp .
\end{equation}
Finally, we adopt the convention of writing $s\grB$ for the set $\{ s\alp:\alp \in \grB\}$.

\begin{theorem}\label{theorem2.1} Suppose that $\bfh\in \dbZ^k$ and 
$\grB\subseteq \dbR$ is measurable. Then one has
\[
I_s(\grB;X;\bfh)\ll X^{-1}(\log X)^sJ_{s+1,k}^*(\grB;2X)^{s/(s+1)}
J_{s+1,k}^*(s\grB;X)^{1/(s+1)}.
\]
\end{theorem}

\begin{proof} Our strategy is based on that underlying the proof of 
\cite[Theorem 2.1]{Woo2012b}, in which the potential for translation-invariance is exploited 
in order to generate an additional variable. Write
\[
\psi(u;\bftet)=\tet_1u+\tet_2u^2+\ldots +\tet_ku^k.
\]
Then for every integral shift $y$ with $0\le y\le X$, one has
\begin{equation}\label{2.4}
f_k(\bfalp;X)=\sum_{y\le x\le X+y}e\left( \psi(x-y;\bfalp)\right) .
\end{equation}
Next write
\begin{equation}\label{2.5}
\grf_y(\bfalp;\gam)=\sum_{0\le x\le 2X}e\left(\psi(x-y;\bfalp)+\gam(x-y)\right)
\end{equation}
and
\[
K(\gam)=\sum_{0\le z\le X}e(-\gam z).
\]
Then we deduce from (\ref{2.4}) via orthogonality that when $0\le y\le X$, one has
\begin{equation}\label{2.6}
f_k(\bfalp;X)=\int_0^1\grf_y(\bfalp;\gam)K(\gam)\d\gam .
\end{equation}

\par We move next to substitute (\ref{2.6}) into (\ref{2.2}) so as to exploit the shift of 
variables by $y$. Define
\begin{equation}\label{2.7}
\grF_y(\bfalp;\bfgam)=\prod_{i=1}^s\grf_y(\bfalp;\gam_i),\quad
\Ktil(\bfgam)=\prod_{i=1}^sK(\gam_i),
\end{equation}
and
\begin{equation}\label{2.8}
\calI(\bfgam;y;\bfh)=\int_\grB \int_{[0,1)^{k-1}}\grF_y(\bfalp;\bfgam)e(-\bfalp\cdot 
\bfh)\d\bfalp .
\end{equation}
Then, when $0\le y\le X$, it follows that
\begin{equation}\label{2.9}
I_s(\grB;X;\bfh)=\int_{[0,1)^s}\calI(\bfgam;y;\bfh)\Ktil(\bfgam)\d\bfgam .
\end{equation}
For the sake of concision, write $\d\bfalp_{k-1}$ for $\d\alp_1\d\alp_2\cdots \d\alp_{k-1}$. 
Then by orthogonality, one finds from (\ref{2.5}) and (\ref{2.7}) that
\begin{equation}\label{2.10}
\int_{[0,1)^{k-1}}\grF_y(\bfalp;\bfgam)e(-\bfalp \cdot \bfh)\d\bfalp_{k-1}=
\sum_{0\le \bfx\le 2X}\Del(\alp_k,\bfgam;\bfh,y),
\end{equation}
where $\Del(\alp_k,\bfgam;\bfh,y)$ is equal to
\[
e\biggl( \sum_{i=1}^s\left(\alp_k(x_i-y)^k+\gam_i(x_i-y)\right) -\alp_kh_k\biggr) ,
\]
when
\begin{equation}\label{2.11}
\sum_{i=1}^s(x_i-y)^j=h_j\quad (1\le j\le k-1),
\end{equation}
and otherwise $\Del(\alp_k,\bfgam;\bfh,y)$ is equal to $0$.\par

By applying the binomial theorem, one sees that whenever the system (\ref{2.11}) is 
satisfied for the $s$-tuple $\bfx$, then
\[
\sum_{i=1}^sx_i^j=sy^j+\sum_{l=0}^{j-1}\binom{j}{l}h_{j-l}y^l\quad (1\le j\le k-1),
\]
and
\[
\sum_{i=1}^sx_i^k=sy^k+\sum_{l=1}^{k-1}\binom{k}{l}h_{k-l}y^l+
\sum_{i=1}^s(x_i-y)^k.
\]
Adopt the convention that $h_0=s$ and write
\[
\grg_y(\bfalp;\bfh;\bfgam)=e\biggl( -\sum_{j=1}^k\alp_j\sum_{l=0}^j\binom{j}{l}
h_{j-l}y^l-y\sum_{i=1}^s\gam_i\biggr).
\]
Then it follows from (\ref{2.7}) and (\ref{2.10}) that
\begin{equation}\label{2.12}
\int_{[0,1)^{k-1}}\grF_y(\bfalp;\bfgam)e(-\bfalp \cdot \bfh)\d\bfalp_{k-1}=
\int_{[0,1)^{k-1}}\grF_0(\bfalp;\bfgam)\grg_y(\bfalp ;\bfh;\bfgam)\d\bfalp_{k-1}.
\end{equation}

\par Observe next that as a consequence of (\ref{2.9}), when $X\in \dbN$, one has
\[
(X+1)I_s(\grB;X;\bfh)=\sum_{0\le y\le X}\int_{[0,1)^s}\calI(\bfgam;y;\bfh)
\Ktil(\bfgam)\d\bfgam .
\]
Thus, from (\ref{2.8}) and (\ref{2.12}) we obtain the relation
\begin{equation}\label{2.13}
I_s(\grB;X;\bfh)\ll X^{-1}\int_{[0,1)^s}|H(\bfgam)\Ktil(\bfgam)|\d\bfgam,
\end{equation}
where
\begin{equation}\label{2.14}
H(\bfgam)=\int_\grB \int_{[0,1)^{k-1}}\grF_0(\bfalp;\bfgam)G(\bfalp;\bfh;\bfgam)\d\bfalp ,
\end{equation}
and
\begin{equation}\label{2.15}
G(\bfalp;\bfh;\bfgam)=\sum_{0\le y\le X}\grg_y(\bfalp;\bfh;\bfgam).
\end{equation}

\par We begin the investigation of the relation (\ref{2.13}) by bounding $H(\bfgam)$. Thus, 
by applying H\"older's inequality on the right hand side of (\ref{2.14}), we deduce that
\begin{equation}\label{2.16}
H(\bfgam)\ll I_1^{s/(s+1)}I_2^{1/(s+1)},
\end{equation}
where
\begin{equation}\label{2.17}
I_1=\int_\grB\int_{[0,1)^{k-1}}|\grF_0(\bfalp;\gam)|^{1+1/s}\d\bfalp
\end{equation}
and
\begin{equation}\label{2.18}
I_2=\int_\grB \int_{[0,1)^{k-1}}|G(\bfalp;\bfh;\bfgam)|^{s+1}\d\bfalp .
\end{equation}

\par On recalling (\ref{2.7}), an application of H\"older's inequality to (\ref{2.17}) yields
\[
I_1\le \prod_{i=1}^s\biggl( \int_\grB\int_{[0,1)^{k-1}}
|\grf_0(\bfalp;\gam_i)|^{s+1}\d\bfalp \biggr)^{1/s}.
\]
Moreover, it follows from (\ref{2.1}) and (\ref{2.5}) via a change of variable that
\[
\int_\grB \int_{[0,1)^{k-1}}|\grf_0(\bfalp;\gam_i)|^{s+1}\d\bfalp =
\int_\grB \int_{[0,1)^{k-1}}|f_k(\bfalp;2X)|^{s+1}\d\bfalp .
\]
Thus, in view of (\ref{2.3}), we have $I_1\le J_{s+1,k}^*(\grB;2X)$. Also, for suitable 
polynomials $\nu_j(y;\bfh)$ with leading term $sy^j$ $(1\le j\le k)$, we find from 
(\ref{2.15}) that
\[
G(\bfalp;\bfh;\bfgam)=\sum_{0\le y\le X}e\biggl( -\biggl( \alp_1\nu_1(y;\bfh)+\ldots 
+\alp_k\nu_k(y;\bfh)+y\sum_{i=1}^s\gam_i\biggr) \biggr).
\]
Then it follows from (\ref{2.18}) via a change of variable that
\[
I_2=\int_\grB \int_{[0,1)^{k-1}}|f_k(s\bfalp;X)|^{s+1}\d\bfalp =
s^{-1}J_{s+1,k}^*(s\grB;X).
\]
By substituting these estimates for $I_1$ and $I_2$ into (\ref{2.16}), we obtain the bound
\begin{equation}\label{2.19}
H(\bfgam)\ll \left( J_{s+1,k}^*(\grB;2X)\right)^{s/(s+1)}\left( 
J_{s+1,k}^*(s\grB;X)\right)^{1/(s+1)}.
\end{equation}

\par We are almost at the end of the proof. All that remains is to recall that
\[
\int_0^1|K(\gam)|\d\gam \ll \int_0^1\min \{ X,\|\gam \|^{-1}\}\d\gam \ll \log X,
\]
whence by (\ref{2.7}), we see that
\[
\int_{[0,1)^s}|\Ktil(\bfgam)|\d\bfgam \ll (\log X)^s.
\]
On substituting the latter bound with (\ref{2.19}) into (\ref{2.13}), we conclude that
\begin{align*}
I_s(\grB;X;\bfh)&\ll X^{-1}\biggl( \sup_{\bfgam \in [0,1)^s}|H(\bfgam)|\biggr) 
\int_{[0,1)^s}|\Ktil(\bfgam)|\d\bfgam \\
&\ll X^{-1}J_{s+1,k}^*(\grB;2X)^{s/(s+1)}J_{s+1,k}^*(s\grB;X)^{1/(s+1)}(\log X)^s.
\end{align*}
This completes the proof of the theorem.
\end{proof}

We now extract from Theorem \ref{theorem2.1} an estimate of minor arc type. When 
$1\le Q\le X$, we define a one-dimensional Hardy-Littlewood dissection as follows. We define 
the set of major arcs $\grM(Q)$ to be the union of the arcs
\[
\grM(q,a)=\{ \alp \in [0,1):|q\alp -a|\le QX^{-k}\},
\]
with $0\le a\le q\le Q$ and $(a,q)=1$, and then write $\grm(Q)=[0,1)\setminus \grM(Q)$ for 
the corresponding set of minor arcs.\par

Before announcing the next lemma, we introduce a Weyl exponent relevant to our discussion. 
When $k$ is an integer with $k\ge 2$, we define the exponent $\sig=\sig(k)$ by taking
\[
\sig(k)^{-1}=\begin{cases} 2^{k-1},&\text{when $2\le k\le 5$,}\\
k(k-1),&\text{when $k\ge 6$.}\end{cases}
\]

\begin{lemma}\label{lemma2.2} Suppose that $k\ge 2$ and $1\le Q\le X$. Then uniformly in 
$(\alp_1,\ldots ,\alp_{k-1})\in \dbR^{k-1}$, one has
\[
\sup_{\alp_k\in \grm(Q)}|f_k(\bfalp;X)|\ll X^{1+\eps}Q^{-\sig}.
\]
\end{lemma}

\begin{proof} Let $\alp_k\in \grm(Q)$. Then by Dirichlet's approximation theorem, there 
exist $a\in \dbZ$ and $q\in \dbN$ with $0\le a\le q\le Q^{-1}X^k$ and $(a,q)=1$ for 
which $|q\alp_k-a|\le QX^{-k}$. Since $\alp_k\not\in \grM(Q)$, we may assume that $q>Q$. 
When $2\le k\le 5$, it now follows from 
Weyl's inequality (see \cite[Lemma 2.4]{Vau1997}) that
\[
f_k(\bfalp;X)\ll X^{1+\eps}(q^{-1}+X^{-1}+qX^{-k})^{2^{1-k}}\ll 
X^{1+\eps}Q^{-\sig(k)}.
\]
On the other hand, when $k\ge 6$, we may apply Vinogradov's methods in their most 
modern incarnations to see that
\[
f_k(\bfalp;X)\ll X^{1+\eps}(q^{-1}+X^{-1}+qX^{-k})^{1/(k(k-1))}\ll 
X^{1+\eps}Q^{-\sig(k)}.
\]
The reader may extract this last conclusion from \cite[Theorem 5.2]{Vau1997} by following 
the proof of \cite[Theorem 1.5]{Woo2012a}, substituting the now confirmed bound 
(\ref{1.3}) for the version of Vinogradov's mean value theorem utilised in the latter.
\end{proof}

\begin{lemma}\label{lemma2.3} Suppose that $\bfh\in \dbZ^k$ and $1\le Q\le X$. Then for 
$s\ge k(k+1)$,
\[
I_s(\grm(Q);X;\bfh)\ll X^{s-\frac{1}{2}k(k+1)+\eps}Q^{-\sig(k)}.
\]
\end{lemma}

\begin{proof} Recall the definition (\ref{1.2}) and write $w=k(k+1)/2$. Then, provided that 
$s\ge k(k+1)$, it follows from (\ref{2.3}) via a trivial estimate for $f_k(\bfalp;X)$ that
\[
J_{s+1,k}^*(\grm(Q);2X)\ll X^{s-k(k+1)}\Bigl( \sup_{\alp_k\in \grm(Q)}|f_k(\bfalp;2X)|
\Bigr) J_{w,k}(X).
\]
Since the now confirmed bound (\ref{1.3}) shows that $J_{w,k}(X) \ll X^{w+\eps}$, we 
deduce from Lemma \ref{lemma2.2} that
\begin{equation}\label{2.20}
J_{s+1,k}^*(\grm(Q);2X)\ll X^{1+\eps}Q^{-\sig(k)}\cdot X^{s-\frac{1}{2}k(k+1)+\eps}.
\end{equation}
We bound $J_{s+1,k}^*(s\grm(Q);X)$ in a similar manner, noting that an elementary 
exercise confirms that $s\grm(Q)\subseteq \grm(Q/s)\mmod{1}$. Thus we find that
\[
\sup_{\alp_k\in s\grm(Q)}|f_k(\bfalp;X)|\ll X^{1+\eps}Q^{-\sig(k)},
\]
and just as in the previous discussion, we infer that
\begin{equation}\label{2.21}
J_{s+1,k}^*(s\grm(Q);X)\ll X^{1+\eps}Q^{-\sig(k)}\cdot X^{s-\frac{1}{2}k(k+1)+\eps}.
\end{equation}

\par By substituting the estimates (\ref{2.20}) and (\ref{2.21}) into Theorem 
\ref{theorem2.1}, we obtain
\[
I_s(\grm(Q);X;\bfh)\ll X^{-1}(\log X)^sX^{s+1-\frac{1}{2}k(k+1)+\eps}Q^{-\sig(k)}.
\]
The conclusion of the lemma follows on noting our convention concerning $\eps$.
\end{proof}

\section{The Hardy-Littlewood dissection} In this section we explain our application of the 
Hardy-Littlewood method in pursuit of the asymptotic formula delivered by Theorem 
\ref{theorem1.1}. We now fix $k$ and $X=2\max\{n_1,n_2^{1/2},\ldots ,n_k^{1/k}\}$, and 
henceforth abbreviate the exponential sum $f_k(\bfalp;X)$ introduced in (\ref{2.1}) simply to 
$f(\bfalp)$. When $\grA\subseteq [0,1)^k$ is measurable, we define the mean value 
$T_s(\grA)=T_s(\grA;X;\bfn)$ by
\begin{equation}\label{3.1}
T_s(\grA;X;\bfn)=\int_\grA f(\bfalp)^se(-\bfalp \cdot \bfn)\d\bfalp .
\end{equation}

\par Our application of the Hardy-Littlewood method requires some discussion concerning the 
associated infrastructure. When $1\le Z\le X$, we denote by $\grK(Z)$ the union of the arcs
\[
\grK(q,\bfa;Z)=\{ \bfalp\in [0,1)^k: \text{$|\alp_j-a_j/q|\le ZX^{-j}$ $(1\le j\le k)$}\},
\]
with $1\le q\le Z$, $0\le a_j\le q$ $(1\le j\le k)$ and $(q,a_1,\ldots ,a_k)=1$, and we put 
$\grk(Z)=[0,1)^k\setminus \grK(Z)$. We have already defined a one-dimensional 
Hardy-Littlewood dissection of $[0,1)$ into sets of arcs $\grM=\grM(Q)$ and 
$\grm=\grm(Q)$. We now fix $L=X^{1/(8k^2)}$ and $Q=L^k$, and we define a 
$k$-dimensional set of arcs by taking $\grN=\grK(Q^2)$ and $\grn=\grk(Q^2)$. This 
intermediate Hardy-Littlewood dissection can be refined to obtain a dissection into a narrower 
set of major arcs $\grP=\grK(L)$ and a corresponding set of minor arcs $\grp=\grk(L)$. In 
this latter dissection, for the sake of concision, it is useful to write 
$\grP(q,\bfa)=\grK(q,\bfa;L)$.\par

We partition the set of points $(\alp_1,\ldots ,\alp_k)$ lying in $[0,1)^k$ into four disjoint 
subsets, namely
\begin{align*}
\grW_1&=[0,1)^{k-1}\times \grm,\\
\grW_2&=\left( [0,1)^{k-1}\times \grM\right)\cap \grn\\
\grW_3&=\left( [0,1)^{k-1}\times \grM\right)\cap (\grN\setminus \grP),\\
\grW_4&=\grP.
\end{align*}
We note in this context that $\grP\subseteq [0,1)^{k-1}\times \grM$, since whenever
\[
|\alp_k-a_k/q|\le LX^{-k}\quad \text{and}\quad 1\le q\le L,
\]
one has
\[
|q\alp_k-a_k|\le L^2X^{-k}<QX^{-k}\quad \text{and}\quad 1\le q<Q.
\]
Thus $\grW_4=\left( [0,1)^{k-1}\times \grM\right)\cap \grP$, and it follows that 
$[0,1)^k=\grW_1\cup \ldots \cup \grW_4$. We therefore deduce via orthogonality that
\begin{equation}\label{3.2}
A_{s,k}(\bfn)=T_s\left([0,1)^k\right)=\sum_{i=1}^4 T_s(\grW_i).
\end{equation}

\par The work of \S2 already permits us to announce a satisfactory upper bound for the 
contribution of the set of arcs $\grW_1$ within (\ref{3.2}).

\begin{lemma}\label{lemma3.1} Whenever $\bfn\in \dbZ^k$ and $s\ge k(k+1)$, one has
\[
T_s(\grW_1)\ll X^{s-\frac{1}{2}k(k+1)-1/(8k^3)}.
\]
\end{lemma}

\begin{proof} By substituting $Q=X^{1/(8k)}$ into Lemma \ref{lemma2.3}, noting that 
$\sig(k)>1/k^2$ for $k\ge 3$, the conclusion of the lemma is immediate.
\end{proof}

\section{Another minor arc estimate} Our next task is to bound the contribution of the set of 
arcs $\grW_2$ within (\ref{3.2}). We begin by providing a familiar estimate of Weyl 
type for the exponential sum $f(\bfalp)$ of strength sufficient for our purposes.

\begin{lemma}\label{lemma4.1} One has 
\[
\sup_{\bfalp \in \grn}|f(\bfalp)|\ll X^{1-1/(6k^2)}\quad \text{and}\quad 
\sup_{\bfalp \in \grp}|f(\bfalp)|\ll X^{1-1/(12k^3)}.
\]
\end{lemma}

\begin{proof} In order to confirm the first bound, we put $\tau=1/(6k^2)$ and 
$\del=1/(4k)$. Since $\tau^{-1}>4k(k-1)$ and $\del>k\tau$, we find from 
\cite[Theorem 1.6]{Woo2012a} that whenever $|f(\bfalp)|\ge X^{1-\tau}$, there exist 
integers $q,a_1,\ldots ,a_k$ such that $1\le q\le X^\del$ and $|q\alp_j-a_j|\le X^{\del-j}$ 
$(1\le j\le k)$. In particular, we see that $q\le Q^2$ and $|\alp_j-a_j/q|\le Q^2X^{-j}$ 
$(1\le j\le k)$, and hence $\bfalp \in \grN\mmod{1}$. We therefore infer that whenever 
$X$ is sufficiently large in terms of $k$, and $\bfalp\in \grn$, then one must have 
$|f(\bfalp)|\le X^{1-\tau}$, and the first conclusion of the lemma follows.\par

For the second bound we put $\tau=1/(12k^3)$ and $\del=1/(8k^2)$. We again have 
$\tau^{-1}>4k(k-1)$ and $\del>k\tau$, and so the same argument applies {\it mutatis 
mutandis}. We now find that whenever $|f(\bfalp)|\ge X^{1-\tau}$, then 
$\bfalp \in \grP\mmod{1}$. Thus, when $X$ is sufficiently large in terms of $k$, and 
$\bfalp\in \grp$, then one must have $|f(\bfalp)|\le X^{1-\tau}$, and the second conclusion 
of the lemma follows.
\end{proof}

In order to obtain a satisfactory bound for the contribution of $\grW_2$, it suffices to 
combine the Weyl estimate provided by Lemma \ref{lemma4.1} with the observation that 
$\grW_2$ has small measure.

\begin{lemma}\label{lemma4.2} Whenever $\bfn\in \dbZ^k$ and $s\ge k(k+1)$, one has
\[
T_s(\grW_2)\ll X^{s-\frac{1}{2}k(k+1)-1/(16k)}.
\]
\end{lemma}

\begin{proof} We begin with an auxiliary estimate. Recall the definition (\ref{1.2}) and write 
$v=k(k-1)/2$. Then, by orthogonality, the mean value
\[
I(\alp_k)=\int_{[0,1)^{k-1}}|f_k(\bfalp;X)|^{2v}\d\bfalp_{k-1} 
\]
counts the integral solutions of the system of equations
\[
\sum_{i=1}^v(x_i^j-y_i^j)=0\quad (1\le j\le k-1),
\]
with $1\le \bfx,\bfy\le X$, each solution being counted with the unimodular weight 
$e\left( \alp_k(x_1^k-y_1^k+\ldots +x_v^k-y_v^k)\right)$. By making use of the bound 
(\ref{1.3}) (see \cite{BDG2016, Woo2016, Woo2019}), we therefore find that 
$I(\alp_k)\ll J_{v,k-1}(X)\ll X^{v+\eps}$, uniformly in $\alp_k$.\par

Observe next that since $\grW_2=\left( [0,1)^{k-1}\times \grM\right)\cap \grn$, it follows 
from (\ref{3.1}) via the triangle inequality that
\[
T_s(\grW_2)\ll \Bigl( \sup_{\bfalp \in \grn}|f(\bfalp)|\Bigr)^{s-k(k-1)}\int_\grM 
I(\alp_k)\d\alp_k.
\]
Since $\text{mes}(\grM)\ll Q^2X^{-k}$, we conclude from Lemma \ref{lemma4.1} that
\begin{align*}
T_s(\grW_2)&\ll X^{s-k(k+1)}\Bigl( X^{1-1/(6k^2)}\Bigr)^{2k}X^{v+\eps}
\text{mes}(\grM)\\
&\ll \left( Q^2X^{-1/(3k)}\right) X^{s-\frac{1}{2}k(k+1)+\eps}.
\end{align*}
Since $Q^2=X^{1/(4k)}$, the conclusion of the lemma is immediate.
\end{proof}

\section{Pruning: the analysis of $\grW_3$} We take an economical approach to the analysis 
of the term $T_s(\grW_3)$. We begin by announcing a mean value estimate of major arc 
type.

\begin{lemma}\label{lemma5.1}
Suppose that $u>\tfrac{1}{2}k(k+1)+2$. Then one has
\[
\int_\grN |f(\bfalp)|^u\d\bfalp \ll_u X^{u-k(k+1)/2}.
\]
\end{lemma}

\begin{proof} This is essentially \cite[Lemma 7.1]{Woo2017}. Our definition of the major 
arcs $\grN$ sets the parameter $Q$ equal to $X^{1/(8k)}$, whereas in the source cited the 
analogous definition is tantamount to setting $Q=X^{1/(2k)}$. The conclusion presented here 
is therefore immediate from the latter source, since our major arcs are contained in those 
employed therein.
\end{proof}

\begin{lemma}\label{lemma5.2} When $\bfn\in \dbZ^3$ and $s\ge k(k+1)$, one has
\[
T_s(\grW_3)\ll X^{s-\frac{1}{2}k(k+1)-1/(12k^3)}.
\]
\end{lemma}

\begin{proof} Since $\grW_3\subseteq \grN\setminus \grP$, we have
\[
\sup_{\bfalp\in \grW_3}|f(\bfalp)|\le \sup_{\bfalp \in \grp}|f(\bfalp)|.
\]
Thus, taking $u=\tfrac{1}{2}k(k+1)+3$, it follows from the triangle inequality that
\[
T_s(\grW_3)\le X^{s-u-1}\Bigl( \sup_{\bfalp\in \grp}|f(\bfalp)|\Bigr) \int_\grN 
|f(\bfalp)|^u\d\bfalp .
\]
Hence, by employing Lemmata \ref{lemma4.1} and \ref{lemma5.1} we see that
\[
T_s(\grW_3)\ll X^{s-1-\frac{1}{2}k(k+1)}\cdot X^{1-1/(12k^3)},
\]
and the conclusion of the lemma follows.
\end{proof}

\section{The major arc contribution} By substituting the conclusions of Lemmata 
\ref{lemma3.1}, \ref{lemma4.2} and \ref{lemma5.2} into (\ref{3.2}), we find that 
whenever $s\ge k(k+1)$, one has
\begin{equation}\label{6.1}
A_{s,k}(\bfn)=T_s(\grW_4)+o(X^{s-\frac{1}{2}k(k+1)}).
\end{equation}
The proof of Theorem \ref{theorem1.1} will be completed by an analysis of 
$T_s(\grW_4)=T_s(\grP)$.\par

Recall (\ref{1.5}) and (\ref{1.6}). Then, when $\bfalp\in \grP(q,\bfa)\subseteq \grP$, write
\[
V(\bfalp;q,\bfa)=q^{-1}S(q,\bfa)I(\bfalp-\bfa/q;X).
\]
Define the function $V(\bfalp)$ to be $V(\bfalp;q,\bfa)$ when $\bfalp\in \grP(q,\bfa)
\subseteq \grP$, and to be $0$ otherwise. Then, when 
$\bfalp\in \grP(q,\bfa)\subseteq \grP$, we see from \cite[Theorem 7.2]{Vau1997} that
\[
f(\bfalp)-V(\bfalp;q,\bfa)\ll q+X|q\alp_1-a_1|+\ldots +X^k|q\alp_k-a_k|\ll L^2,
\]
whence, uniformly for $\bfalp\in \grP$, we have the bound
\[
f(\bfalp)^s-V(\bfalp)^s\ll X^{s-1+1/(4k^2)}.
\]
Thus, since $\text{mes}(\grP)\ll L^{2k+1}X^{-k(k+1)/2}$, we deduce that
\begin{equation}\label{6.2}
\int_\grP f(\bfalp)^se(-\bfalp \cdot \bfn)\d\bfalp =\int_\grP V(\bfalp)^s
e(-\bfalp \cdot \bfn)\d\bfalp +o(X^{s-\frac{1}{2}k(k+1)}).
\end{equation}

\par Next, write
\[
\Ome=[-LX^{-1},LX^{-1}]\times \cdots \times [-LX^{-k},LX^{-k}].
\]
Then one finds that
\begin{equation}\label{6.3}
\int_\grP V(\bfalp)^se(-\bfalp \cdot \bfn)\d\bfalp =\grS(X)\grJ(X),
\end{equation}
where
\[
\grJ(X)=\int_\Ome I(\bfbet ;X)^se(-\bfbet \cdot \bfn)\d\bfbet
\]
and
\[
\grS(X)=\sum_{1\le q\le L}\sum_{\substack{1\le \bfa\le q\\ (q,a_1,\ldots ,a_k)=1}}
q^{-s}S(q,\bfa)^se_q(-\bfa \cdot \bfn).
\]

\par When $s>\tfrac{1}{2}k(k+1)+1$, the singular integral $\grJ_{s,k}(\bfn)$ defined in 
(\ref{1.7}) converges absolutely (see \cite[Theorem 1.3]{AKC2004}), and in particular 
$\grJ_{s,k}(\bfn)\ll 1$. Also, by \cite[Theorem 7.3]{Vau1997}, one has
\[
I(\bfbet;X)\ll X(1+|\bet_1|X+\ldots +|\bet_k|X^k)^{-1/k}.
\]
Write $h_j=n_jX^{-j}$ $(1\le j\le k)$. Then after two changes of variable, we obtain
\begin{align}
\grJ(X)&=X^{s-\frac{1}{2}k(k+1)}\int_{\dbR^k}I(\bfbet;1)^se(-\bfbet \cdot \bfh)\d\bfbet 
+o(X^{s-\frac{1}{2}k(k+1)})\notag \\
&=\left( \grJ_{s,k}(\bfn)+o(1)\right) X^{s-\frac{1}{2}k(k+1)}.\label{6.4}
\end{align}

\par Similarly, by reference to \cite[Theorem 2.4]{AKC2004}, one sees that the singular series 
$\grS_{s,k}(\bfn)$ defined in (\ref{1.8}) converges absolutely for $s>\tfrac{1}{2}k(k+1)+2$, 
and in particular $\grS_{s,k}(\bfn)\ll 1$. Also, it follows from \cite[Theorem 7.1]{Vau1997} 
that when $(q,a_1,\ldots ,a_k)=1$, one has $S(q,\bfa)\ll q^{1-1/k+\eps}$. Thus,
\begin{equation}\label{6.5}
\grS(X)=\grS_{s,k}(\bfn)+o(1).
\end{equation}

\par By substituting (\ref{6.4}) and (\ref{6.5}) into (\ref{6.3}), and thence into (\ref{6.2}), 
we obtain
\[
\int_\grP f(\bfalp)^se(-\bfalp \cdot \bfn)\d\bfalp =\grS_{s,k}(\bfn)\grJ_{s,k}(\bfn)
X^{s-\frac{1}{2}k(k+1)}+o(X^{s-\frac{1}{2}k(k+1)}).
\]
Thus, on recalling (\ref{3.1}), we find that
\[
T_s(\grW_4)=T_s(\grP)=\grS_{s,k}(\bfn)\grJ_{s,k}(\bfn)X^{s-\frac{1}{2}k(k+1)}
+o(X^{s-\frac{1}{2}k(k+1)}),
\]
and by substituting this relation into (\ref{6.1}), we conclude that
\[
A_{s,k}(\bfn)=\grS_{s,k}(\bfn)\grJ_{s,k}(\bfn)X^{s-\frac{1}{2}k(k+1)}+o(X^{s-\frac{1}{2}
k(k+1)}).
\]
This confirms the conclusion of Theorem \ref{theorem1.1}.

\section{Other applications} The method underlying the proof of Theorem \ref{theorem2.1} 
may be applied in many similar situations. Thus, the system of equations (\ref{1.1}) may be 
replaced by
\begin{equation}\label{7.1}
\sum_{i=1}^lx_i^j-\sum_{i=l+1}^{l+m}x_i^j=n_j\quad (1\le j\le k),
\end{equation}
provided that $l\ne m$ and $s=l+m\ge k(k+1)$. The Hilbert-Kamke problem corresponds to 
the situation here where $m=0$. What is critical is that the underlying system (\ref{7.1}) is 
{\it not} translation-dilation invariant, even in the special situation in which 
$\bfn={\mathbf 0}$. Let $B_{l,m,k}(\bfn;X)$ denote the number of solutions of the system 
(\ref{7.1}) with $1\le x_i\le X$ $(1\le i\le s)$. Then provided that $s\ge k(k+1)$, methods 
almost identical to those of this paper establish an asymptotic formula
\[
B_{l,m,k}(\bfn;X)\sim CX^{s-\frac{1}{2}k(k+1)},
\]
in which $C$ is an appropriate product of local densities. Here, the major innovation lies with 
the subconvex minor arc estimate provided by an analogue of Theorem \ref{theorem2.1}. 
The major arc analysis is handled in earlier work \cite{Ark1984}.\par

When $l=m=k(k+1)/2$, the system (\ref{7.1}) lacks translation-dilation invariance when 
$(n_1,\ldots ,n_{k-1})\ne {\mathbf 0}$, though a shadow of this invariance property 
remains. The latter complicates any attempt to derive an analogue of Theorem 
\ref{theorem2.1}. We have more to say concerning such systems in the memoir 
\cite{Woo2021a}.\par

Finally, when the coefficients $c_i\in \dbZ\setminus \{0\}$ $(1\le i\le k(k+1))$ satisfy the 
condition $c_1+\ldots +c_s\ne 0$, the methods underlying the proof of Theorem 
\ref{theorem2.1} remain in play when one examines the system
\[
c_1x_1^j+\ldots +c_sx_s^j=n_j\quad (1\le j\le k).
\]
Indeed, this scenario was discussed in work of the author dating from 2015 that was the 
subject of talks presented in G\"oteborg, Oxford and Strobl-am-Wolfgangsee (see 
\cite{Woo2021b}).

\bibliographystyle{amsbracket}
\providecommand{\bysame}{\leavevmode\hbox to3em{\hrulefill}\thinspace}

\end{document}